\documentclass[11pt]{article}

\usepackage{graphicx}
\usepackage{url}
\usepackage{amsmath}
\usepackage{amsfonts}
\usepackage{verbatim}
\usepackage{amssymb}
\usepackage{amsthm}

\usepackage[all]{xy}

\usepackage{color}
\definecolor{red}{rgb}{1,0,0}

\newtheorem{thm}{Theorem}[section]

\newtheorem{prop}[thm]{Proposition}
\newtheorem{cor}[thm]{Corollary}
\newtheorem{lem}[thm]{Lemma}

\newtheorem{ex}[thm]{Example}
\newtheorem{quest}[thm]{Question}
\newtheorem{obs}[thm]{Observation}

\def\ord#1{| #1 |}

\newcommand{\diag}{{\rm diag}}
\newcommand{\rank}{\operatorname{rank}}

\newcommand{\bff}{{\bf f}}
\newcommand{\bv}{{\bf v}}

\newcommand{\bu}{{\bf u}}
\newcommand{\bx}{{\bf x}}

\newcommand{\bz}{{\bf z}}

\newcommand{\be}{{\bf e}}

\newcommand{\R}{\mathbb{R}}

\newcommand{\Rnn}{\R^{n\times n}}
\newcommand{\Rn}{\R^{v}}
\newcommand{\Rsn}{\R^{(v)}}

\newcommand{\G}{\mathcal{G}}
\newcommand{\RR}{\mathcal{R}}
\newcommand{\sym}{\mathcal{S}}
\newcommand{\MM}{\mathcal{M}}
\newcommand{\TT}{\mathcal{T}}
\newcommand{\TRA}{\TT_{\RR_A}}
\newcommand{\NSA}{\NN_{\sym_A}}
\newcommand{\TSA}{\TT_{\sym_A}}
\newcommand{\NRA}{\NN_{\RR_A}}
\newcommand{\NN}{\mathcal{N}}

\newcommand{\mr}{\operatorname{mr}}

\newcommand{\amr}{\operatorname{amr}}

\newcommand{\M}{\operatorname{M}}

\newcommand{\nul}{\operatorname{null}}
\newcommand{\Span}{\operatorname{span}}

\newcommand{\exi}{\bar\xi}
\newcommand{\emr}{\overline{\mr}}

\newcommand{\x}{\times}
\newcommand{\bit}{\begin{itemize}}
\newcommand{\eit}{\end{itemize}}
\newcommand{\ben}{\begin{enumerate}}
\newcommand{\een}{\end{enumerate}}
\newcommand{\beq}{\begin{equation}}
\newcommand{\eeq}{\end{equation}}
\newcommand{\bea}{\begin{eqnarray*}}
\newcommand{\eea}{\end{eqnarray*}}
\newcommand{\bean}{\begin{eqnarray}}
\newcommand{\eean}{\end{eqnarray}}
\newcommand{\bpf}{\begin{proof}}
\newcommand{\epf}{\end{proof}}

\newcommand{\noi}{\noindent}

\newcommand{\E}{{\bf E}}
\newcommand{\textdef}{\textit}

\title{Expected values of parameters associated with the minimum rank of a graph\thanks{Research of Hall, Hogben, Shader partially supported by American Institute of Mathematics SQuaRE,``Minimum Rank of Symmetric Matrices described by a Graph." Martin's research partially supported by NSA Grant H-98230-08-1-0015.}}

\author{Tracy Hall\thanks{Department of Mathematics, Brigham Young University, Provo UT 84602 (h.tracy@gmail.com).}  \and  Leslie Hogben\thanks{Department of Mathematics, Iowa State University, Ames, IA 50011, USA (LHogben@iastate.edu) and American Institute of Mathematics, 360 Portage Ave,
Palo Alto, CA 94306, USA (hogben@aimath.org).}\and Ryan Martin\thanks{Department of Mathematics, Iowa State University, Ames, IA 50011, USA (rymartin@iastate.edu).} \and  Bryan  Shader\thanks{Department of Mathematics, University of Wyoming, Laramie, WY 82071, USA
(bshader@uwyo.edu).}}

\begin{document}
\maketitle

\begin{abstract}
{We investigate the expected value  of various graph parameters associated with the minimum rank of a graph, including minimum rank/maximum nullity 
and related Colin de Verdi\`ere-type parameters.  Let $G(v,p)$ denote the usual Erd\H{o}s-R\'enyi random graph on $v$ vertices with edge probability $p$.  We obtain bounds for the expected value of the random variables $\mr(G(v,p))$, $\M(G(v,p))$, $\nu(G(v,p))$ and $\xi(G(v,p))$, which yield bounds on the average values of these parameters over all labeled graphs of order $v$.} \end{abstract}

\noi {\bf AMS subject classifications:} 05C50; 05C80; 15A03 

\noi {\bf Keywords:} minimum rank; maximum nullity; average minimum rank; average maximum nullity; expected value; Colin de Verdi\`ere type parameter; positive semidefinite minimum rank; delta conjecture; rank; matrix; random graph; graph

%%%%%%%%%%%%%%%%%%%%%%%%%%%%%%%%%%%%%%%%%%
\section{Introduction}\label{sintro}

%$\lfloor$

The set of $v\times v$ real symmetric matrices will be denoted by $\Rsn$.  For $A\in \Rsn$,
the {\em graph} of $A$, denoted $\G(A)$, is the graph with vertices
$\{1,\dots,v \}$  and edges $\{ \{i,j \} : ~a_{ij} \ne 0,  1 \le i <j \le v \}$.  Note that the diagonal of $A$ is ignored in determining $\G(A)$.
%removed () around A
The {\em minimum rank} of a graph $G$ on $v$ vertices  is  $$\mr(G)=\min\{ \rank A : A \in \Rsn, \G(A)=G\}.$$
The {\em maximum nullity} or {\em maximum corank} of a graph $G$  is  
%corrected min to max and removed () 
$$\M(G)=\max\{ \nul A : A \in \Rsn, \G(A)=G\}.$$  Note that $$\mr(G)+\M(G)=v.$$
Here a {\em graph} is a pair $G=(V(G),E(G))$, where $V$ is the (finite, nonempty) set of vertices and $E$ is the set of edges (an edge is a two-element subset of vertices); what we call a graph is sometimes called a simple undirected graph. We use the notation $v(G)=|V(G)|$ and $e(G)=|E(G)|$.

The {\em minimum rank problem} (of a graph, over the real numbers) is to  determine $\mr(G)$ for any graph $G$. See  \cite{FH}  for a survey of known results and discussion of the motivation for  the minimum rank problem; an extensive bibliography is also provided there.  The  minimum rank problem was
a focus of the 2006 workshop ``Spectra of families of matrices described by graphs, digraphs, and sign patterns"  held at the American Institute of Mathematics \cite{AIM}. One of the questions raised during
the workshop was:

\begin{quest}
What is the average minimum rank of a graph on $v$ vertices?
\end{quest}

Formally, we define the {\em average minimum rank} of graphs of order $v$ to be the sum over all labeled graphs  of order $v$ of the minimum ranks of  the graphs,   divided by the number of (labeled) graphs of order $v$.  That is,
$$\amr(v)=\frac {\sum_{v(G) = v} \mr(G)}{2^{v \choose 2}}.$$

Let $G(v,p)$ denote the Erd\H{o}s-R\'enyi random graph on $v$ vertices with edge probability $p$.  That is, every pair of vertices is adjacent, independently, with probability $p$.  Note that 
for %(added)
$G(v,1/2)$, %gives that 
every labeled $v$-vertex graph is equally likely (each labeled graph is chosen with probability $2^{-\binom{v}{2}}$), so
$$ \amr(v)=E\left[\mr(G(v,1/2))\right] . $$

Our goal in this paper is to determine statistics about the random variable $\mr(G(v,p))$ and other related parameters.  We highlight the two main results of this paper by focusing on the $p=1/2$ case:
\begin{thm}\label{main}
   Given $\amr(v)=E\left[\mr(G(v,1/2))\right]$, then for $v$ sufficiently large,
   \begin{enumerate}
      \item \label{it:main1} $\displaystyle\left|\mr(G(v,1/2))-\amr(v)\right|<\sqrt{v\ln\ln v}$ with probability approaching 1 as $v\rightarrow\infty$, and
      \item \label{it:main2} $0.146907v<\amr(v)<0.5v+\sqrt{7v\ln v}$.
   \end{enumerate}
\end{thm}

% revised due to section 6 change
In general, we show that the random variable $\mr(G(v,p))$ is tightly concentrated around its mean (Section
\ref{sconcent}), and establish lower and upper bounds for its expected value in Sections \ref{sv9} and
\ref{sMdelta}.  We also establish  an upper bound on the Colin de Verdi\`ere type parameter  $\xi(G)$, which is related to $\M(G)$, in  Section \ref{sxi} (the definition of $\xi$ is given in that section).    This bound is used in Section \ref{snew} to establish bounds on the expected value of the random variable $\xi(G(v,p))$.  The upper bound on $\xi(G(v,p))$ may lead to a better upper bound on the expected value of $\M(G(v,p))$ and hence a better lower bound on the expected value of $\mr(G(v,p))$.

%%%%%%%%%%%%%%%%%%%%%%%%%%%%%%%%%%%%%%%%%%
\section{Tight concentration of expected minimum rank}\label{sconcent}

% corrected wording as per RM e-mail 18-Jan-10
Although we are unable to determine precisely the mean of $\mr(G(v,p))$, in this section we show that this random variable is tightly concentrated around its mean, and  thus $\mr(G(v,1/2))$ is tightly concentrated around the average
minimum rank.

 A {\em martingale} is a sequence of random variables $X_0,\dots,X_{v-1}$ such that
$$\E[X_{i+1} | X_i,X_{i-1},\dots,X_1]=X_i.$$
 The martingale we use is the vertex exposure martingale (as described on pages 94-95 of  \cite{AS}) for the graph parameter $f(G)=\frac 1 2 \mr(G)$  (the factor $\frac 1 2$ is needed because deletion of a vertex may change the minimum rank by $2$; see Corollary \ref{tightmr} below).  $G(v,p)$ is sampled to obtain a specific graph $H$, and $X_i$ is the expected value of the graph parameter $f(G)=\frac 1 2 \mr(G)$ when the neighbors of vertices $v_1,\dots,v_i$ are known.  Since nothing is known for $X_0$, $X_0=\E[f(G(v,p))]=\frac 1 2\E[\mr(G(v,p))]$. Since the entire graph $H$ is revealed at stage $v-1$, $X_{v-1}=\frac 1 2\mr(H)$.

  The method for showing tight concentration uses Azuma's inequality for martingales (see Section 7.2 of~\cite{AS}) and was pioneered by Shamir and Spencer~\cite{SS}.  The following corollary of Azuma's inequality is used.

\begin{thm} {\rm  \cite[Corollary 7.2.2]{AS}} Let $b=X_0,\dots,X_v$ be a martingale with
$$\ord{X_{i+1}-X_i}\le 1$$
for all $0\le i\le v$.  Then
$$\Pr[|X_v-b|>\beta\sqrt v]<2e^{-\beta^2/2}.$$
\end{thm}

The proof that derives the tight concentration of the chromatic number of the random graph \cite[Theorem 7.2.4]{AS} from  \cite[Corollary 7.2.2]{AS} via the vertex exposure martingale remains valid for any graph parameter $f(G)$ such that when $G$ and $H$ differ only in the exposure of a single vertex, then $\ord{f(G) - f(H)} \le 1$. 
\begin{thm}
Let $p\in (0,1)$.  Let $f$ be a graph invariant such that for any graphs $G$ and $H$, if
% replace E by \E in next line
  $x\in V(G)=V(H)$  and $G-x = H-x$, then $\ord{f(G) - f(H)} \le 1$.  Let $\mu=\E\left[f(G(v,p))\right]$.  Then, for any $\beta>0$,
$$ \Pr\left[\left|f(G(v,p))-\mu\right|>\beta\sqrt{v-1}\right]<2e^{-\beta^2/2} .$$ \label{thm:mgale}
\end{thm}

\begin{cor}\label{tightmr}
% replace E by \E in next line and added "let"
Let $p\in (0,1)$ be fixed and let 
$\mu=\E\left[\mr(G(v,p))\right]$.  For any $\beta>0$,
$$ \Pr\left[\left|\mr(G(v,p))-\mu\right|>2\beta\sqrt{v-1}\right]<2e^{-\beta^2/2} . $$ In particular,
$$ \left|\mr(G(v,p))-\mu\right|<\sqrt{v\ln\ln v} $$
with probability approaching $1$ as $v\rightarrow\infty$. \label{cor:concen}
\end{cor}
\bpf
It is well-known that for any graph $G$ and any vertex $x\in V(G)$,
$ 0\leq\mr(G)-\mr(G-x)\leq 2.$  Thus if $V(H)=V(G)$ and $G-x = H-x$,  then $\ord{\mr(G) - \mr(H)} \le 2$. 
For the first statement, apply Theorem~\ref{thm:mgale} with $f(G)=\frac 1 2 \mr(G)$.  For the second statement, let $\beta=\frac 1 2 \sqrt{\ln\ln v}$ and conclude
$$  \Pr\left[\left|\mr(G(v,p))-\mu\right|>\sqrt{v\ln\ln v}\right]<2\left(\frac 1 {{\ln v}}\right)^{1/8}.$$
\epf

Note that Corollary~\ref{tightmr} gives the result in Theorem~\ref{main}(\ref{it:main1}).

%%%%%%%%%%%%%%%%%%%%%%%%%%%%%%%%%%%%%%%%%%
\section{Observations on  parameters of random graphs}\label{sdegree}

Large deviation bounds easily show that the degree sequence of the random graph is tightly concentrated.  In this section, we provide some well-known results that will be used later.  The version of the Chernoff-Hoeffding bound that we use is given in~\cite{AS}.
\begin{thm}{\rm \cite[Theorem A.1.16]{AS}}
   Let $X_i$, $1\leq i\leq n$, be mutually independent random variables with all $\E[X_i]=0$ and all $|X_i|\leq 1$.  Set $S=X_1+\cdots+X_n$.  Then for any $a>0$, $$\Pr[S>a]<\exp\{-a^2/(2n)\}.$$ \label{thm:chernoff}
\end{thm}

\vspace{-3mm}
It is well-known that Theorem~\ref{thm:chernoff} can be applied to the number of edges in a random graph: 

\begin{thm}
   Let $p$ be fixed and let $G$ be distributed according to $G(v,p)$.  Then,
   $$ e(G)\leq p\binom{v}{2}+v\sqrt{2\ln v}, $$
   with probability at least $1-v^{-2}$.  In addition, $e(G)\geq p\binom{v}{2}-v\sqrt{2\ln v}$ with probability at least $1-v^{-2}$. \label{thm:edges}
\end{thm}

\begin{proof}
   Let $G$ be distributed according to the random variable $G(v,p)$.
   We may regard \\$\left\{\{x,y\}\in E(G) : x\neq y\right\}$ to be $\binom{v}{2}$ mutually independent indicator random variables.  Subtract $p$ from each and they become random variables with mean $0$ and magnitude at most 1.  Using Theorem~\ref{thm:chernoff}, we see that $\Pr\left[e(G)-p\binom{v}{2}>a\right]<\exp\left\{-a^2/\left(2\binom{v}{2}\right)\right\}$.

   Choose $a=v\sqrt{2\ln v}$; we see that
   $$ e(G)-p\binom{v}{2}\leq v\sqrt{2\ln v} , $$
   with probability at least $1-v^{-2}$.  By multiplying the random variables above by $-1$, we obtain
   $$ e(G)-p\binom{v}{2}\geq -v\sqrt{2\ln v}, $$
   with probability at least $1-v^{-2}$.
\end{proof}

 Let $\delta(G)$ (respectively, $\Delta(G)$) denote the minimum (maximum) degree of a vertex of $G$.  Theorem~\ref{thm:chernoff} can also be applied to the neighborhood of each vertex to give bounds on $\delta(G)$ and $\Delta(G)$.

\begin{thm}
   Let $p$ be fixed and let $G$ be distributed according to $G(v,p)$.  Then,
   $$  pv-\sqrt{6v\ln v} \le \delta(G)\le \Delta(G)\leq pv+\sqrt{6v\ln v}$$
   with probability at least $1-2v^{-2}$.  \label{thm:degree}
\end{thm}

\begin{proof}
   Let $G$ be distributed according to the random variable $G(v,p)$.
   For each $x\in V(G)$, we may regard $\left\{\{x,y\}\in E(G) : y\neq x\right\}$ to be $v-1$ mutually independent indicator random variables.  Using Theorem~\ref{thm:chernoff}, we see that $\Pr\left[\left|\deg(x)-p(v-1)\right|>a\right]<2\exp\{-a^2/(2(v-1))\}$.

   Thus, the probability that there exists a vertex with degree that deviates by more than $a$ from $p(v-1)$ is at most
   $$ v\times 2\exp\{-a^2/(2(v-1))\} . $$

   Choose $a=\sqrt{6v\ln v}$ and we see that, simultaneously for all $x\in V(G)$,
   $$ \left|\deg(x)-pv\right|\leq \sqrt{6v\ln v}, $$
   with probability at least $1-2v^{-2}$.
   \end{proof}

%%%%%%%%%%%%%%%%%%%%%%%%%%%%%%%%%%%%%%%%%%
\section{A lower bound for expected minimum rank}\label{sv9}

In this section we show that if $v$ is sufficiently large, then the expected value of $\mr(G(v,p))$ is at least $c(p) v + o(v)$, where $c(p)$ is the solution to equation (\ref{eq:c1}) below.  In the case $p=1/2$, $c(p)\approx 0.1469077$, so the average minimum rank   is greater than $0.146907v$ for $v$ sufficiently large.

% ital zero pattern
The {\em zero-pattern} $\zeta(\bx)$ of the real vector $\bx = (x_1,\ldots, x_\ell)$ is the $(0, *)$-vector obtained
from $\bx$ by replacing its nonzero entries by $*$.   The \textit{support} of  the zero pattern $\bz=(z_1,\dots,z_\ell)$ is the set $S(\bz)=\{i : z_i\ne 0\}$. We %will 
modify the proof of Theorem 4.1 from \cite{RBG} to obtain the following  result.

\begin{thm}\label{RBGthm} If $\bff(\bx) = (f_1(\bx), f_2(\bx),\ldots, f_m(\bx))$ is
an $m$-tuple of polynomials in $n$ variables over a field $F$  with $m\ge n$, each $f_i$ of degree at most $d$, then the number of zero-patterns $\bz=\zeta(\bff(\bx))$  with $\ord{S(\bz)}\le s$ is at most
$$ {n+sd  \choose n} .\label{RBGeq}$$
\end{thm}

\begin{proof}
We follow the proof  in \cite{RBG}. Assume that the $m$-tuple ${\bf f}=(f_1,\ldots,f_m)$ of polynomials over field $F$ has the $M$ zero-patterns $\bz_1,\dots,\bz_M$. Choose  $\bu_1,\ldots,\bu_M\in F^n$ such that $\zeta(\bff(\bu_i))=\bz_i$.

 Set
$$ g_i=\prod_{k\in S(\bz_i)}f_k . $$
Note that
$$ g_i(\bu_j)\neq 0\qquad\mbox{if and only if}\qquad S(\bz_i)\subseteq S(\bz_j). $$

We show that polynomials $g_1,\ldots,g_M$ are linearly independent.  Assume on the contrary that there is a nontrivial linear combination $\sum_{i=1}^M\beta_ig_i=0$, where each $\beta_i\in F$.  Let $j$ be a subscript such that $|S(\bz_j)|$ is minimal among the $S(\bz_i)$ with $\beta_i\neq 0$, so for every $i$ such that  $i\ne j$ and $\beta_i\ne 0$, $S(\bz_i)\not\subseteq S(\bz_j)$.  So substituting $\bu_j$ into the linear combination  gives $\beta_j g_j(\bu_j)= 0$, a contradiction.

Thus, $g_1,\ldots,g_M$ are linearly independent over $F$.  Each $g_i$ has degree at most $sd$ and the dimension of the space of polynomials of degree $\leq D$ is exactly $\binom{n+D}{n}$.
\end{proof}

By Sylvester's Law of Inertia, every real symmetric $v\times v$ matrix of rank at most $r$ can be expressed in  the form $X^TD_iX$ for some $i$ such that $0\le i \le r$, where  $D_i=\diag(1,\dots,1,-1,\dots,-1)$ is an $r\x r$ diagonal matrix with  $i$ diagonal entries equal to $1$ and $r-i$ equal to $-1$ and $X$ is an $r\times v$ real matrix.   There are $r+1$ diagonal matrices $D_i$.  Let each entry of $X$ be a variable; the  total number of variables is $rv$ and each entry of the matrix $X^TD_iX$ is a  polynomial of degree at most $2$.

 Let $c(p)$ be the solution to \begin{equation}\label{eq:c1}
\frac{(c+p)^{2c+2p}}{(c)^{2c}(p)^{2p}}p^p(1-p)^{(1-p)}=1
\end{equation}
for a fixed value of $p (0<p<1)$.  This equation has a unique solution, because it is equivalent to  $\frac{(c+p)^{2 c+2 p}}{c^{2 c} }=\frac{p^p}{(1-p)^{(1-p)}}$, and for a fixed $p$ and $c\ge 0$, $\frac{(c+p)^{2 c+2 p}}{c^{2 c} }$ is a strictly increasing function of $c$  and $p^{2p}<\frac{p^p}{(1-p)^{(1-p)}}$.  The  values of $c(p), 0<p<1$ are graphed in Figure \ref{figcp}.
\begin{figure}[!h]
\begin{center}
\scalebox{.8}{\includegraphics{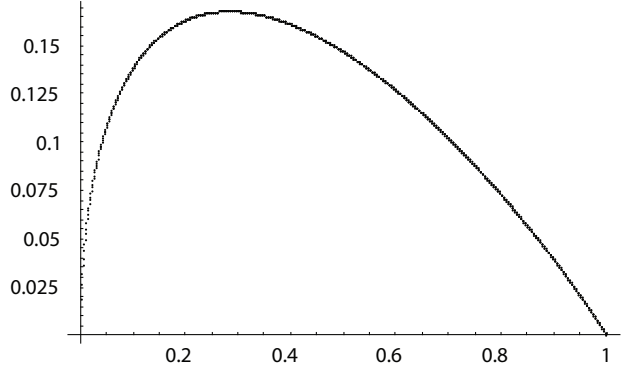}}
\caption{The graph of $c(p)$}\label{figcp}
\end{center}
\end{figure}

\begin{thm}\label{newthm} Let $G$ be distributed according to $G(v,p)$ for a fixed $p$, $0<p<1$.
For any $c<c(p)$,
 the expectation $ \E\left[\mr(G)\right]$ satisfies $$ \E\left[\mr(G)\right]>cv$$
for $v$  sufficiently large.

Furthermore, for any such  $c$,
$\Pr\left[\mr(G(v,p))\leq cv\right]\to 0$ as $v\to\infty$.
\end{thm}
\bpf
 Let $G$ be distributed according to $G(v,p)$.  Let $\mathcal{E}$ be the event that $\left|e(G)-p\binom{v}{2}\right|\leq v\sqrt{2\ln v}$.  By the law of total expectation,
\begin{eqnarray*}
   \E[\mr(G)] & = & \E[\mr(G)\mid {\mathcal{E}}]\Pr[{\mathcal{E}}]+\E[\mr(G)\mid {\mathcal{E}}^c]\Pr[{\mathcal{E}}^c] \\
   & \geq & \E[\mr(G) \mid {\mathcal{E}}]\Pr\left[{\mathcal{E}}\right] \\
   & \geq & (r+1)\Pr[\mr(G)>r\mid {\mathcal{E}}]\Pr[{\mathcal{E}}] \\
   & = & (r+1)\left(1-\Pr[\mr(G)\leq r\mid {\mathcal{E}}]\right)\left(1-\Pr[{\mathcal{E}^c}]\right) \\
   & \geq & (r+1)-(r+1)\Pr[\mr(G)\leq r\mid {\mathcal{E}}]-(r+1)\Pr[{\mathcal{E}}^c] \\
   & \geq & (r+1)-v\Pr[\mr(G)\leq r\mid {\mathcal{E}}]-v\Pr[{\mathcal{E}}^c]
\end{eqnarray*}

Theorem~\ref{thm:edges} shows that $v\Pr[{\mathcal{E}}^c]\leq v^{-1}$.  It remains to bound $\Pr[\mr(G)\leq r\mid {\mathcal{E}}]$.

\begin{eqnarray*}
\Pr\left[\mr(G)\leq r\mid{\mathcal{E}}\right]
& = & \sum_{\scriptsize \begin{array}{l} G : v(G)=v, \mr(G)\leq r \\ \left|e(G)-p\binom{v}{2}\right|\leq
v\sqrt{2\ln v}\end{array}}\Pr[G\in G(v,p)] \\
& = & \sum_{\scriptsize \begin{array}{l} G : v(G)=v, \mr(G)\leq r \\ \left|e(G)-p\binom{v}{2}\right|\leq v\sqrt{2\ln v}\end{array}}p^{e(G)}(1-p)^{\binom{v}{2}-e(G)} \\
& = & \sum_{\scriptsize \begin{array}{l} G : v(G)=v, \mr(G)\leq r \\ \left|e(G)-p\binom{v}{2}\right|\leq v\sqrt{2\ln v}\end{array}}\left(\frac{p}{1-p}\right)^{e(G)}(1-p)^{\binom{v}{2}}
\end{eqnarray*}

If $p<1/2$, then we use a lower bound for $e(G)$, given $\mathcal{E}$; if $p>1/2$, an upper bound.  So, we can bound the term inside the summation as
$$ \left(\frac{p}{1-p}\right)^{e(G)}(1-p)^{\binom{v}{2}}\leq\left(\frac{\max\{p,1-p\}}{\min\{p,1-p\}}\right)^{v\sqrt{2\ln v}}\left(p^p(1-p)^{(1-p)}\right)^{\binom{v}{2}} . $$
Hence,
\begin{eqnarray*}
\lefteqn{\Pr\left[\mr(G)\leq r\mid{\mathcal{E}}\right]} \\
& \leq &  \left(\frac{\max\{p,1-p\}}{\min\{p,1-p\}}\right)^{v\sqrt{2\ln v}}\left(p^p(1-p)^{(1-p)}\right)^{\binom{v}{2}}\left|\left\{G : v(G)=v, \left|e(G)-p\binom{v}{2}\right|\leq v\sqrt{2\ln v} , \mr(G)\leq r\right\}\right| .
\end{eqnarray*}

The number of $v$ vertex graphs with between $p\binom{v}{2}-v\sqrt{2\ln v}$ and $p\binom{v}{2}+v\sqrt{2\ln v}$ edges and minimum rank at most $r$ is at most the number of $v\times v$ symmetric pattern matrices obtained as $X^TD_iX, i=0,\dots,r$ with $X$ an $r\x v$ matrix for which the cardinality of the support of the superdiagonal entries is at most $p\binom{v}{2}+v\sqrt{2\ln v}$.   We can apply Theorem~\ref{RBGeq} with $n=rv$, $d=2$ and $s\leq p\binom{v}{2}+v\sqrt{2\ln v}$.  Therefore, because there are $r+1$ diagonal matrices,
\begin{eqnarray*}
\lefteqn{\Pr\left[\mr(G)\leq r\mid{\mathcal{E}}\right]} \\
& \leq &  \left(\frac{\max\{p,1-p\}}{\min\{p,1-p\}}\right)^{v\sqrt{2\ln v}}\left(p^p(1-p)^{(1-p)}\right)^{\binom{v}{2}}(r+1)\binom{rv+2p\binom{v}{2}+2v\sqrt{2\ln v}}{rv} .
\end{eqnarray*}

By  Corollary \ref{bincoeffcor}  in Appendix \ref{sapp}, for fixed $c$ and $p$ with $r=cv$,
$$ \binom{rv+2p\binom{v}{2}+2v\sqrt{2\ln v}}{rv}
\le \left((1+o(1))\left(\frac{(c+p)^{c+p}}{c^cp^p}\right)\right)^{v^2} . $$
Thus
$$ \Pr[\mr(G)\leq cv\mid {\mathcal{E}}]\leq\left((1+o(1))\frac{(c+p)^{2c+2p}}{(c)^{2c}(p)^{2p}}p^p(1-p)^{(1-p)}\right)^{v^2/2} . $$

As long $c<c(p)$ %(for $c(p)$ defined in $(\ref{eq:c1})$)
and $v$ is sufficiently large, the quantity $v\Pr\left[\mr(G)\leq r\mid {\mathcal{E}}\right]$ is less than 1, giving
$$ \E[\mr(G)]\geq (r+1)-v\Pr[\mr(G)\leq r\mid {\mathcal{E}}]-v\Pr[{\mathcal{E}}^c]>r+1-o(1)\ge r. $$

Furthermore, as long as $c<c(p)$, $ \Pr\left[\mr(G)\leq cv\mid {\mathcal{E}}\right]\to 0$ as $v\to \infty$, and by Theorem \ref{RBGthm}, $\Pr[{\mathcal{E}}^c]\to 0$ as $v\to \infty$.  Since
$$\Pr[\mr(G)\leq cv] \le  \Pr[\mr(G)\leq cv\mid {\mathcal{E}}] + \Pr[{\mathcal{E}}^c],$$
 $\Pr[\mr(G)\le cv]\to 0$ as $v\to\infty$.
\epf

\begin{cor}\label{coravmrlow}  For $v$ sufficiently large, the average minimum rank over all labeled graphs of order $v$   satisfies $$ \amr(v)>0.146907v.$$
Furthermore, if $G$ is chosen at random from all labeled graphs of order $v$,
$\Pr[\mr(G)\leq 0.146907v]\to 0$ as $v\to\infty$.
\end{cor}
\bpf For $p=1/2$, $ \E[\mr(G)]=\amr(v)$ and  $0.146907<c(p)$.
\epf

Note that Corollary~\ref{tightmr} gives the lower bound in Theorem~\ref{main}(\ref{it:main2}).  We note further the lack of symmetry with respect to $p$.  The value $c(p)$ approaches zero as $p$ approaches zero, which is not the case with the upper bound that we describe in the next section.

%%%%%%%%%%%%%%%%%%%%%%%%%%%%%%%%%%%%%%%%%%
\section{An upper bound for expected minimum rank}\label{sMdelta}

In this section we show that if $v$ is sufficiently large, then the expected value of $\mr(G(v,p))$ is at most $ (1-p)v+\sqrt{7v\ln v}$.  Thus  the average minimum rank for graphs of order $v$ is at most $0.5v+\sqrt{7v\ln v}$.

Let $\kappa(G)$ denote the \textdef{vertex connectivity} of $G$.  That is, if $G$ is not complete, it is the smallest number $k$ such that there is a set of vertices $S$, with $|S|=k$, for which $G-S$ is disconnected. By convention, $\kappa(K_v)=v-1$.

Following the terminology of \cite{LSS89}, for a graph $G$ an \textdef{orthogonal representation of $G$ of dimension $d$} is a set of vectors in $\R^d$, one corresponding to each vertex, such that if two vertices are nonadjacent, then their corresponding vectors are orthogonal.  Every graph has an orthogonal representation in any dimension by associating the zero vector with every vertex.
A \textdef{faithful orthogonal representation of  $G$ of dimension $d$} is a set of vectors in $\R^d$, one corresponding to each vertex, such that two (distinct) vertices are nonadjacent if and only if their corresponding vectors are orthogonal. Note that in the minimum rank literature, the term {``orthogonal representation"} is customarily used for what is here called a faithful orthogonal representation.

The following result of Lov\'asz, Saks and Schrijver \cite{LSS89} (see also the note on errata, \cite{LSS00} Theorem 1.1) is the basis for an upper bound for minimum rank.

\begin{thm}\label{LSSthm}{\rm \cite[Corollary 1.4]{LSS89}} Every graph $G$ on $v$ vertices has a faithful orthogonal representation of dimension $v-\kappa(G)$.
\end{thm}

Let $\mr_{+}(G)$ denote the minimum rank among all symmetric positive semidefinite matrices $A$ such that $\G(A)=G$, and let $\M_{+}(G)$ denote the maximum nullity among all such matrices.
 It is well known (and easy to see) that every faithful orthogonal representation of dimension $d$ gives rise to a positive semidefinite matrix of rank $d$ and vice versa.

\begin{cor}\label{cor:kappa} For any graph $G$ on $v$ vertices,
\beq\label{kappa} \mr(G)\leq\mr_{+}(G)\leq v-\kappa(G),
\eeq
or equivalently,
\beq\label{kappaM} \kappa(G)\le\M_{+}(G)\le \M(G).
\eeq
\end{cor}

%changed 1/2 to p in next line
Our  proof of the upper bound on the expected value of $\mr(G(v,p))$ uses the bound (\ref{kappa})  and the relationship (on average) between  the connectivity $\kappa(G)$ and the minimum degree $\delta(G)$.
At the AIM workshop \cite{AIM} it was conjectured that for any graph $G$, $\delta(G)\le M(G)$, or equivalently $\mr(G)\le v(G)-\delta(G)$ \cite{AIMrpt}.  The conjecture was proved for bipartite graphs in \cite{delta} but remains open in general.   In \cite{LSS89} it is reported that in 1987, Maehara made a conjecture equivalent to $\mr_{+}(G)\leq v(G)-\delta(G) $, which would imply $\mr(G)\leq v(G)-\delta(G) $.

\begin{thm}\label{avdelta} Let $G$ be distributed according to $G(v,p)$. For $v$ sufficiently large, the expected value of minimum rank satisfies $\E[\mr(G)]\le (1-p)v+\sqrt{7v\ln v}$.

For $v$ sufficiently large, the
  average minimum rank over all labeled graphs of order $v$   satisfies
 $$ \amr(v)\le 0.5v+\sqrt{7v\ln v}.$$
\end{thm}
\bpf
In \cite{BT85} (see also section 7.2 of \cite{BolRGT}), Bollob\'as and Thomason prove that if $G$ is distributed according to $G(v,p)$, then $\Pr[\kappa(G)=\delta(G)]\rightarrow 1$ as $v\rightarrow\infty$, without any restriction on $p$.  Lemma~\ref{connectivity} in  Appendix \ref{sappB} shows that for $p$ fixed and $v$ large enough, $\Pr[\kappa(G)<\delta(G)]\leq 3v^{-2}$.  Let $\mathcal{E}$ be the event that $\kappa(G)=\delta(G)$ and $\delta(G)\geq pv-\sqrt{6v\ln v}$.  For $G$ distributed according to $G(v,p)$, the law of total expectation gives
\begin{eqnarray*}
   \E[\kappa(G)] & = & \E[\kappa(G)\mid {\mathcal{E}}]\Pr[{\mathcal{E}}]+\E[\kappa(G)\mid {\mathcal{E}}^c]\Pr[{\mathcal{E}}^c] \\
   & \geq & \left(pv-\sqrt{6v\ln v}\right)\left(1-\Pr[{\mathcal{E}}^c]\right) \\
   & \geq & pv-\sqrt{6v\ln v}-v\left(\Pr[\delta(G)<pv-\sqrt{6v\ln v}]+\Pr[\kappa(G)<\delta(G)]\right).
\end{eqnarray*}

We use Theorem \ref{thm:degree} and the result that $v\Pr[\kappa(G)<\delta(G)]\leq 3v^{-1}$ to see that $$ \E\left[\kappa(G)\right]\geq pv-\sqrt{6v\ln v}-2v^{-1}-3v^{-1}\ge pv-\sqrt{7v\ln v} , $$
for $v$ sufficiently large.  By (\ref{kappa}), $\E[\mr(G)]\leq (1-p)v+\sqrt{7v\ln v}$.
\epf

Theorem \ref{avdelta} gives the upper bound in Theorem~\ref{main}(\ref{it:main2}).  Note that Theorem \ref{avdelta} actually establishes $\E[\mr_+(G)]\le (1-p)v+\sqrt{7v\ln v}$.  Since $\mr(G)\le\mr_+(G)$ for any graph $G$, the lower bound in Theorem \ref{newthm} in Section \ref{sv9} certainly bounds $\E[\mr_+(G)]$ from below.  %No improvement in the lower bound results by assuming that the matrices are positive semidefinite (the only change is the replacement of the $r+1$ diagonal matrices $D_i,i=0,...,r$ by $D_r$).

%%%%%%%%%%%%%%%%%%%%%%%%%%%%%%%%%%%%%%%%%%
%major editorial revision including combining old section 6 and 7 due to Thm 6.1 having been proved by Hein.

\section{Bounds for  $\nu(G)$ and $\xi(G)$}\label{sxi}

 In this section we discuss the the Colin de Verdi\`ere type parameters $\nu(G)$ and $\xi(G)$, and establish an upper bound on $\xi(G)$ in terms of the number of edges of the graph.    This upper bound, and a known lower bound for $\nu(G)$, have implications for the average value of $\nu$ and $\xi$ (see Section \ref{snew}).

In 1990 Colin de Verdi\`ere
(\cite{CdV} in English) introduced the graph parameter $\mu$ that is equal to the
maximum multiplicity of eigenvalue 0 among all matrices satisfying several conditions including the  Strong Arnold
Hypothesis (defined below).  
The parameter $\mu$, which is used to characterize planarity, is the  first of
several parameters that require the Strong Arnold Hypothesis and bound the
maximum nullity from below  (called {\em Colin de Verdi\`ere type parameters}).  All the Colin de Verdi\`ere type parameters we discuss have been shown to be  minor monotone.

 A {\em contraction} of  $G$ is obtained by identifying two adjacent vertices
of
$G$, deleting any loops  that arise in this process, and replacing any multiple edges by a single edge.
 A {\em minor}
of $G$ arises by performing a  sequence of deletions of edges, deletions of isolated
vertices, and/or contractions of edges.
 A graph parameter $\beta$ is {\em minor monotone} if for any minor $G'$  of $G$, $\beta(G') \le \beta(G)$.

A symmetric real matrix $M$ is said to satisfy the {\em Strong Arnold
Hypothesis (SAH)} provided there does {not} exist a nonzero real symmetric matrix
$X$ satisfying
$AX = 0$,
 $A\circ X = 0$, and
 $I\circ X=0$,
where $\circ$ denotes the Hadamard (entrywise) product and $I$ is the identity matrix.

The SAH is equivalent to the requirement that certain manifolds intersect transversally.  Specifically, for $A=[a_{ij}]\in\Rsn$ let   $$\RR_A=\{B\in\Rsn : \rank B = \rank A\},$$ and
 $$\sym_A = \{B\in\Rsn : \G(B)=\G(A)\} .$$
 Then   $\RR_A$ and $\sym_A$  intersect transversally  at $A$ if and only if  $A$ satisfies the SAH (see \cite{HLS}).

Another minor monotone parameter,  introduced by Colin de Verdi\`ere
in \cite{CdVnu}, is denoted by $\nu(G)$ and defined to be the maximum nullity
among matrices $A$ that satisfy:
\begin{enumerate}
\item  $\G(A )=G$;
\item $A$ is positive semidefinite;
\item $A$ satisfies the Strong Arnold Hypothesis.
\end{enumerate}
Clearly $\nu(G)\le \M_+(G)$.

 The parameter $\xi(G)$  was introduced in \cite{BFH3} as a Colin de Verdi\`ere type parameter intended for use in computing maximum nullity and minimum rank, by removing any unnecessary restrictions while preserving minor monotonicity. % (the SAH seems to be necessary to obtain minor monotonicity). 
 Define $\xi(G)$ to be the maximum multiplicity of 0 as an eigenvalue
among matrices $A\in\Rsn$ that satisfy:
\begin{itemize}
\item  $\G(A )=G$.
\item $A$ satisfies the Strong Arnold Hypothesis.
\end{itemize}
%In \cite{BFH3} it is shown that  the parameter  $\xi(G)$ is minor monotone.  
  Clearly, $\nu(G) \le \xi(G)\le \M(G)$. 
%  An orthogonal representation $\ba_1,\dots \ba_v\in\R^d$ of $G$  is in {\em general position} if every subset of $d$ vectors is linearly independent.  
  The following lower bound on $\nu(G)$ has been established by van der Holst using the results of Lov\'asz, Saks and Schrijver.

  \begin{thm}\label{kappanu}{\rm \cite[Theorem 4]{vdH08}}  For every graph $G$,
  $$\kappa(G)\le\nu(G)\le\xi(G).$$
  \end{thm}

The following bound on the Colin de Verdi\`ere number $\mu$ in terms of the number of edges $e(G)$ is given in  \cite{P} for any connected graph $G\ne K_{3,3}:$
$$e(G) \ge  \frac {\mu(G)(\mu(G)+1)} 2.$$
We will show that for any connected graph $G$,
$$e(G) + b\ge \frac {\xi(G)(\xi(G)+1)} 2$$
where $b=1$ if $G$ is bipartite and $b=0$ otherwise.
%This bound has implications for the average value of $\xi$ (see Section \ref{snew}).

For a manifold $\MM$ and matrix $A\in\MM$, let $\TT_{\MM_A}$ be the tangent space in $\Rsn$ to $\MM$ at $A$
 and let $\NN_{\MM_A}$ be the normal (orthogonal complement) to $\TT_{\MM_A}$.

\begin{obs}\label{obs1}{\rm  \cite[p. 9]{HLS}} $\null$
\ben
\item $\TSA=\{B: \forall i\ne j, a_{ij}=0\Rightarrow b_{ij}=0\}$.
\item\label{c2} $\NSA=\{ X: \forall i, x_{ii}=0 \mbox{ and } \forall i\ne j, a_{ij}\ne 0\Rightarrow x_{ij}=0\}$.
\item\label{c1} $\TRA=\{WA+AW^T: W\in\Rnn \}=\{B\in\Rsn: \bv^TB\bv=0~ \forall \bv\in\ker A\}$.
\item\label{c3} $\NRA=\Span(\{\bv\bv^T: \bv\in\ker A\})=\{X\in \Rsn:AX = 0\}$.
\een
\end{obs}

 Clearly $\dim \TSA = e(G)+v$.  These  observations can also be used to provide the exact dimension of $\NRA$ and thus of  $\TRA$.

\begin{prop}\label{dimN} Let $A\in\Rsn$ 
and let  $\bu_1,\dots,\bu_q$ be an orthonormal basis for $\ker A$.  Then $U=\{\bu_i\bu_i^T:1\le i\le q\}\cup \{\bu_i\bu_j^T+\bu_j\bu_i^T:1\le i<j\le q\}$ is a basis for $ \Span(\{\bv\bv^T: \bv\in\ker A\}).$  Thus $\dim \NRA=\frac {q(q+1)} 2$.
\end{prop}
\bpf  Let $\NN=\Span(\{\bv\bv^T: \bv\in\ker A\})$.  Since $\bu_i\bu_j^T+\bu_j\bu_i^T=(\bu_i+\bu_j)(\bu_i+\bu_j)^T-\bu_i\bu_i^T-\bu_j\bu_j^T$, $U\subset \NN.$

Show $U$ spans $\NN$:
$$\left(\sum_{i=1}^q s_i\bu_i\right)\left(\sum_{j=1}^q s_j\bu_i\right)^T=\sum_{i=1}^q\sum_{j=1}^q s_is_j\bu_i\bu_j^T=\sum_{i=1}^q s_i^2\bu_i\bu_i^T+\sum_{1\le i<j\le q} s_is_j(\bu_i\bu_j^T+\bu_j\bu_i^T)$$

Show $U$ is linearly independent: Let  $Y=\sum_{i=1}^q s_i\bu_i\bu_i^T+\sum_{1\le i<j\le q} s_{ij}(\bu_i\bu_j^T+\bu_j\bu_i^T)$ and suppose $Y=0.$  For any $k$, $0=\bu_k^TY\bu_k=s_k$ and for $\ell<k$, $0=\bu_\ell Y\bu_k=s_{\ell k}$, and $U$ is linearly independent.
\epf

\begin{cor}\label{dimT}  $\dim \TRA=v\rank A-\frac {\rank A(\rank A-1)} 2$.
\end{cor}
\bpf  Let $\rank A=r$. By  Observation \ref{obs1} and Proposition \ref{dimN},
$$\dim\TRA=\dim\Rsn-\dim \NRA = \frac {v(v+1)} 2-\frac{(v-r)(v-r+1)} 2. \qedhere$$
\epf

 An {\em optimal matrix for $\xi(G)$} is a matrix $A$ such that $\G(A)=G, \nul A=\xi(A)$, and $A$ has the Strong Arnold Hypothesis.
\begin{thm}\label{thm:xiubd} Let $G$ be a connected graph.
\beq\label{eq0}e(G) + b\ge \frac {\xi(G)(\xi(G)+1)} 2\eeq
where $b=1$ if $G$ is bipartite and every optimal matrix for $\xi(G)$ has zero diagonal, and $b=0$ otherwise.
\end{thm}
\bpf  Let $A$ be an optimal matrix for $\xi(G)$, chosen to have at least one nonzero diagonal entry if there is such an optimal matrix.  Let $\rank A=r$.

The Strong Arnold Hypothesis for $A$ is $ \NRA\cap \NSA =\{0\}  $, which is equivalent by taking orthogonal complements to
$$
\TRA+ \TSA = \Rsn
$$
Therefore
\bea\dim \TRA + \dim \TSA-\dim (\TRA\cap\TSA) &= & \dim \Rsn\\
 vr-\frac {r(r-1)} 2 + e(G) + v-\dim (\TRA\cap\TSA) &= & \frac {v(v+1)} 2\eea
\bea
e(G) & = & \frac {v(v+1)} 2 - vr+\frac {r(r-1)} 2- v + \dim (\TRA\cap\TSA) \\
& = & \frac 1 2((v-r)^2 +(v-r)) - v + \dim (\TRA\cap\TSA) \\
 & = & \frac{\xi(G)(\xi(G)+1)} 2 - v + \dim (\TRA\cap\TSA)
\eea
Thus
$$\frac{\xi(G)(\xi(G)+1)} 2 = e(G) + v - \dim (\TRA\cap\TSA).$$

Let $D=\diag(d_1,\dots,d_n)$ be a  diagonal matrix.  Then by Observation \ref{obs1}.\ref{c1},
$DA+AD\in\TRA$. Clearly,  $DA+AD\in\TSA$, so $DA+AD\in \TRA\cap\TSA$.  Let $\be_k$ be the $k$th standard basis  vector of $\Rn$.  Define $D_k=\diag(\be_k)$ and  $B_k=D_kA+AD_k$.
Note that %$(B_k)_{i}=\delta_{ki} a_{ii}$ and 
$(B_k)_{ij}=(\delta_{ki}+\delta_{kj}) a_{ij}$, where $\delta_{ii}=1$ and $\delta_{ij}=0$ for $i\ne j$.

  We show first that if $\displaystyle \sum_{k=1}^{v}c_k B_k=0 \mbox{ and } c_t=0 \mbox{ for some } t \mbox{ such that } 1\le t\le v,$
 then $c_k=0$ for all  $ 1\le k\le v$.
 For every neighbor $y$ of $t$,
$$ 0=\left(\sum_{k=1}^{v}c_k B_k\right)_{ty}=\sum_{k=1}^{v}c_k (\delta_{kt}+\delta_{ky}) a_{ty}=c_y a_{ty}.$$
Since $\{t,y\}$ is an edge of $G$,  $a_{ty}\ne 0$, and so $c_y=0$.
Since $G$ is connected, every vertex can be reached by a path from $t$, and so
$c_1=\dots=c_{v}=0$.

Since   $$\sum_{k=1}^{v-1}c_k B_k = \sum_{k=1}^{v}c_k B_k \mbox{ with } c_v=0,$$ it follows that for every graph $G$ and $\xi(G)$-optimal matrix $A$ (without any assumption about the diagonal),
the matrices $B_k, k=1,\dots,v-1,$ are linearly independent, and thus
  $$\dim (\TRA\cap\TSA)\ge v-1\qquad\mbox{  and  }\qquad \frac{\xi(G)(\xi(G)+1)} 2 \le e(G) +1,$$

Now suppose that $A$ has a nonzero diagonal entry or $G$ is not bipartite.  We show that 
%there exists $c_t= 0$, so 
the matrices $B_k, k=1,\dots,v$  are linearly independent, so   $$\dim (\TRA\cap\TSA)\ge v\qquad\mbox{  and  }\qquad \frac{\xi(G)(\xi(G)+1)} 2 \le e(G)$$

Let $$\sum_{k=1}^vc_k B_k=0.$$
 If $A$ has a nonzero diagonal entry $a_{tt}$, then
 $ 0=\left(\sum_{k=1}^{v}c_k B_k\right)_{tt} 
 =2c_t a_{tt}$, and so $c_t=0$.
If $G$ is not bipartite, there is an odd cycle; without loss of generality let this odd cycle be $(1,\dots,t)$.   Then for $i=1,\dots, t-1$,
$$ 0=\left(\sum_{k=1}^{v}c_k B_k\right)_{i,i+1} 
=(c_i+c_{i+1}) a_{i,i+1};$$
similarly $0=(c_t+c_{1}) a_{t,1}$.  Since $  \{t,1\}$ and $\{i,i+1\}, i=1,\dots,t-1$ are edges of $G$,
$$ c_i+c_{i+1}=0,i=1,\dots, t-1, \mbox{ and } c_t+c_{1}=0.$$
 By adding equations $(-1)^{i} (c_i+c_{i+1}=0), i=1,\dots,t-1$ to $c_t+c_{1}=0$, we obtain
$ 2 c_t=0.$
 \epf

If $G$ is the disjoint union of its connected components $G_1,\dots,G_h$, then $\xi(G)=\max_{i=1,\dots,h}\{\xi(G_i)\}$ \cite{BFH3}.
\begin{cor}\label{cor:xiubd} For every graph $G$,
$$  \frac {\xi(G)(\xi(G)+1)} 2 \le e(G) + 1.$$
\end{cor}

\begin{ex}{\rm  The complete bipartite graph $K_{3,3}$ demonstrates that $b=1$ is sometimes necessary in the bound (\ref{eq0}), because $\xi(K_{3,3})=4$, so $\frac{\xi(G)(\xi(G)+1)} 2=10$, and  $e(K_{3,3}) = 9$.
}\end{ex}

%%%%%%%%%%%%%%%%%%%%%%%%%%%%%%%%%%%%%%%%%%
\section{Bounds for the expected value of $\xi$}\label{snew}

In this section we show that if $v$ is sufficiently large, then the expected value of $\xi(G(v,p))$ is asymptotically at most $\sqrt p v$.
It follows that the average value of $\xi$ for graphs of order $v$ is asymptotically at most $\frac 1 {\sqrt 2} v$.

We will make the notion of asymptotic expected value more precise, both for minimum rank and for~$\xi$.

Define  
$$\emr(p)=\limsup_{v\to\infty} \frac{\E\left[\mr(G(v,p))\right]} v.$$
(This is a careful definition, as the lim sup is almost certainly a limit.)
In previous sections we have shown that for $0<p<1$,
$$c(p)\le\emr(p)\le 1-p.$$

Now define $$\exi(p)=\liminf_{v\to\infty} \frac{\E\left[\xi(G(v,p))\right]} v.$$
The quantity $\exi(p)$ should be compared to $1-\emr(p)$ rather than $\emr(p)$, since $\xi(G)$ measures a nullity rather than a rank.

Our starting point is an immediate consequence of  Corollary \ref{cor:xiubd}.
\begin{cor} For every graph $G$,
 \bean \xi(G) & \le & \frac 1 2(-1+\sqrt{9+8e(G) }).\label{eq01}\eean
\end{cor}

\begin{cor}\label{coraxi}
 For $0 < p < 1$,
 $$p \le \exi(p)\le \sqrt{p}.$$
\end{cor}
\bpf
The proof that $p \le \exi(p)$ follows from Theorem \ref{kappanu}
by exactly the same reasoning that showed that $\emr(p) \le 1-p$.

From inequality (\ref{eq01}), if $e(G)\ge 2$,
$$\xi(G)  \le  \sqrt{2e(G)}.$$
For a fixed $\epsilon>0$, as $v\to\infty$, almost all graphs
sampled from $G(v,p)$
satisfy
$$ e(G) \le (1+\epsilon)\frac p 2 v(v-1),$$
so almost all graphs satisfy
$$\xi(G)\le \sqrt{2e(G)} \le \sqrt{2(1+\epsilon)\frac p 2 v(v-1)}\le \sqrt{1+\epsilon}\sqrt{p}v.$$
This completes the proof of the second inequality
 $\exi(p)\le \sqrt{p}.$
\epf

Since for every graph $G$, $\xi(G)\le\M(G)$ and for every $v>1$ there exists a  graph $H$ such that $\xi(H)<\M(H)$,
$\E\left[\xi(G(v,p))\right]$ is strictly less than $\E\left[\M(G(v,p))\right]$ for $v>1$ and $0 < p < 1$.
However it is quite possible that taking the limit gives $\exi(p) + \emr(p) = 1$, in which case Corollary \ref{coraxi} would provide a better asymptotic lower bound for expected minimum rank than that given in Corollary \ref{coravmrlow}.  The graphs of these bounds are shown in Figure \ref{figall}.

 \begin{figure}[!h]
\begin{center}
\scalebox{.8}{\includegraphics{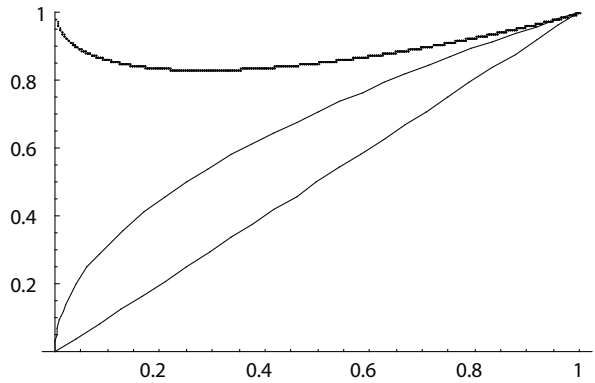}}
\caption{The graphs of $1-c(p)> \sqrt p> p$ for $0<p<1$}\label{figall}
\end{center}
\end{figure}

%%%%%%%%%%%%%%%%%%%%%%%%%%%%%%%%%%%%%%%%%%
\appendix
\section{Appendix: Estimation of the binomial coefficient}\label{sapp}

\begin{lem} \label{lem:bincoeff}
Let $N$ be a positive integer and $\alpha,\beta,\gamma$ be real numbers with $\alpha,\beta\in(0,1)$ and $\gamma\in[0,1]$.  Then,
$$ \binom{(\alpha+\beta+\gamma)N}{\alpha N}\leq
E(\alpha,\beta,\gamma,N)\left(\frac{(\alpha+\beta)^{\alpha+\beta}}{\alpha^\alpha\beta^\beta}\right)^N , $$
where
$$  E(\alpha,\beta,\gamma,N)=\sqrt{\frac{\alpha+\beta}{2\pi\alpha\beta N}}\exp\left\{\frac{1}{12(\alpha+\beta)N}+\gamma\left(1+\frac{\alpha}{\beta}\right)N\right\} . $$
\end{lem}

\begin{proof}
We use Stirling's formula as given in \cite[page 216]{MGT}:
$$ \sqrt{2\pi n}\left(\frac{n}{e}\right)^n\leq n!\leq e^{1/(12n)}\sqrt{2\pi n}\left(\frac{n}{e}\right)^n . $$

From this formula,
\begin{eqnarray*}
\lefteqn{\binom{(\alpha+\beta+\gamma)N}{\alpha N}} \\ & \leq &
e^{1/(12(\alpha+\beta+\gamma)N)} \frac{\sqrt{2\pi(\alpha+\beta+\gamma)N}}{\sqrt{2\pi \alpha N}\sqrt{2\pi(\beta+\gamma)N}}
\left(\frac{(\alpha+\beta+\gamma)N}{e}\right)^{(\alpha+\beta+\gamma)N}
\left(\frac{e}{\alpha N}\right)^{\alpha N}
\left(\frac{e}{(\beta+\gamma)N}\right)^{(\beta+\gamma)N} \\
& \leq &
e^{1/(12(\alpha+\beta)N)} \sqrt{\frac{\alpha+\beta+\gamma}{2\pi \alpha(\beta+\gamma)N}} \frac{\left(\alpha+\beta+\gamma\right)^{(\alpha+\beta+\gamma)N}}
{\alpha^{\alpha N}
\left(\beta+\gamma\right)^{(\beta+\gamma)N}}
\end{eqnarray*}

Since $\frac{\alpha+\beta+\gamma}{\alpha(\beta+\gamma)}\le\frac{\alpha+\beta}{\alpha\beta}$,
\begin{eqnarray*}
\lefteqn{\binom{(\alpha+\beta+\gamma)N}{\alpha N}} \\
& \leq &
e^{1/(12(\alpha+\beta)N)} \sqrt{\frac{\alpha+\beta}{2\pi\alpha\beta N}} \left(\frac{\left(\alpha+\beta+\gamma\right)^{(\alpha+\beta+\gamma)}}
{\alpha^{\alpha}
\left(\beta+\gamma\right)^{(\beta+\gamma)}} \right)^N\\
& \leq &
e^{1/(12(\alpha+\beta)N)} \sqrt{\frac{\alpha+\beta}{2\pi\alpha\beta N}} \left(\frac{\left(\alpha+\beta\right)^{\alpha+\beta}}
{\alpha^{\alpha}\beta^{\beta}}\right)^N
\left(\left(\frac{\alpha+\beta+\gamma}{\alpha+\beta}\right)^{\alpha+\beta}
\left(\frac{\beta}{\beta+\gamma}\right)^\beta
\left(\frac{\alpha+\beta+\gamma}{\beta+\gamma}\right)^\gamma\right)^N \\
& \leq &
\left(\frac{\left(\alpha+\beta\right)^{\alpha+\beta}}
{\alpha^{\alpha}\beta^{\beta}}\right)^N
\sqrt{\frac{\alpha+\beta}{2\pi\alpha\beta N}} e^{1/(12(\alpha+\beta)N)}
 \left(\left(1+\frac{\gamma}{\alpha+\beta}\right)^{\alpha+\beta}
\left(1+\frac{\alpha}{\beta}\right)^\gamma\right)^N \\
\end{eqnarray*}
Because $1+x\le e^x$,
$$\binom{(\alpha+\beta+\gamma)N}{\alpha N} \leq 
\left(\frac{\left(\alpha+\beta\right)^{\alpha+\beta}}
{\alpha^{\alpha}\beta^{\beta}}\right)^N
\sqrt{\frac{\alpha+\beta}{2\pi\alpha\beta N}} e^{1/(12(\alpha+\beta)N)}   \exp\left\{\gamma N
+\gamma\frac{\alpha}{\beta}N\right\}\qedhere
$$
\end{proof}

\begin{cor}\label{bincoeffcor}  Let $p,c$ be fixed and let $r=cv$, with $v\to\infty$.

$$ \binom{rv+2p\binom{v}{2}+2v\sqrt{2\ln v}}{rv}
\le \left((1+o(1))\left(\frac{(c+p)^{c+p}}{c^cp^p}\right)\right)^{v^2} . $$
\end{cor}
\bpf
$$ \binom{rv+2p\binom{v}{2}+2v\sqrt{2\ln v}}{rv}=\binom{c v^2+pv^2 -pv+2v\sqrt{2\ln v}}{cv^2} . $$
Let $N=v^2$, $\alpha=c$, $\beta=p$, and $\gamma=\frac 1 {v^2}(-pv +2v\sqrt{2\ln v})$.  With $p$ and $c$ fixed, by Lemma \ref{lem:bincoeff} we see that
$$ \binom{rv+2p\binom{v}{2}+2v\sqrt{2\ln v}}{rv}
\le \left((1+o(1))\left(\frac{(c+p)^{c+p}}{c^cp^p}\right)\right)^{v^2} . \qedhere$$
\epf

%%%%%%%%%%%%%%%%%%%%%%%%%%%%%%%%%%%%%%%%%%

\section{Appendix: Connectivity is minimum degree}\label{sappB}
Bollob\'as and Thomason \cite{BT85} proved that for $G\sim G(v,p)$, regardless of $p$, then $\Pr[\kappa(G)<\delta(G)]\rightarrow 0$ as $v\rightarrow\infty$.  Bollob\'as \cite{B83} proved the result for $p$ in a restricted interval, but the statement of his theorem is much more general.  For our result, we need to bound the probability that $\kappa(G)=\delta(G)$ where $G\sim G(v,p)$, but  need the result only for a  fixed $p$.

\begin{lem}\label{connectivity}
   Let $p\in (0,1)$ be fixed and $G$ be distributed according to $G(v,p)$.  If $v$ is sufficiently large, then
   $$ \Pr[\kappa(G)<\delta(G)]\leq 3v^{-2} . $$
\end{lem}

\bpf 
 Let $\delta=\delta(G)$.  
By Theorem~\ref{thm:edges} we see that, with probability at least $1-v^{-2}$,
\beq\label{eqappBdelta}
\delta \leq\frac{2e(G)}{v}\leq\frac{2\left(p\binom{v}{2}+v\sqrt{2\ln v}\right)}{v}=p(v-1)+2\sqrt{2\ln v}\leq pv+2\sqrt{2\ln v} . \eeq
For the remainder of the proof we assume $\delta \leq pv+2\sqrt{2\ln v} .$

 If $\kappa(G)<\delta$, then there exists a partition $V(G)=V_1\cup S\cup V_2$ such that $|S|<\delta$, $2\leq |V_1|\leq |V_2|$ and there is no edge between $V_1$ and $V_2$.  Let the {\em closed neighborhood of vertex $x$} be denoted $N[x]$ and be equal to $\{x\}\cup N(x)$.

We will show first that there is an integer $t$ such that the probability that $2\leq |V_1|\leq t$ is at most $v^{-2}$ (we will determine the value of $t$ later).  By a different calculation, we will then show that the probability that $t<|V_1|\leq (v+\delta)/2$ is also at most $v^{-2}$.  Note that we don't attempt to optimize the probability or to give a range of $p$ over which these conditions hold.  A total probability of $2v^{-2}$ is sufficient for our purposes and results in an easier proof.

The event $\{2\leq |V_1|\leq t\}$ can occur only if there are two distinct vertices, $x_1$ and $x_2$, such that the cardinality of the union of their closed neighborhoods is less than $t+\delta$.  
For vertices $y_i\in V(G)\backslash \{x_1,x_2\}$, let $Y_i$ be independent indicator variables for $y_i\in N[x_1]\cup N[x_2]$.  Since the probability $y_i\notin N[x_1]\cup N[x_2]$ is $(1-p)^2$, $$\E[|N[x_1]\cup N[x_2]|]=\E[2+Y_1+\dots+ Y_{v-2}]]=2+(v-2)(2p-p^2).$$
Hence, assuming $(t+\delta)-\left(2+(2p-p^2)(v-2)\right)<0$, by the negative version of Theorem~\ref{thm:chernoff},
\begin{eqnarray*}
   \Pr\left[2\leq |V_1|\leq t\right] & \leq & 
   \binom{v}{2}\lefteqn{\Pr\left[\left|N[x_1]\cup N[x_2]\right|<t+\delta\right]} \\
   & = & \binom{v}{2}\Pr\left[\left|N[x_1]\cup N[x_2]\right|-\left(2+(2p-p^2)(v-2)\right)
   <(t+\delta)-\left(2+(2p-p^2)(v-2)\right)\right] \\
   & \leq & \exp\left\{2\ln v-\frac{1}{2(v-2)}\left((t+\delta)-2-(2p-p^2)(v-2)\right)^2\right\}.
\end{eqnarray*}
Thus if $t+\delta\leq(2p-p^2)(v-2)+2-3\sqrt{v\ln v}$, then $\Pr\left[2\leq |V_1|\leq t\right]<v^{-2}$.

Since $v\geq 2$ and we have assumed $\delta \leq pv+2\sqrt{2\ln v}$, we may set
$$ t=(2p-p^2)v-3\sqrt{v\ln v}-\delta\geq (2p-p^2)v-3\sqrt{v\ln v}-pv-2\sqrt{2\ln v}\geq p(1-p)v-5\sqrt{v\ln v} . $$
We will use the bound $\binom{v}{i}\leq\left(\frac{ev}{i}\right)^i$ which is true for all $1\leq i\leq v$ \cite[page 216]{MGT}, and the trivial bound $\binom{v-i}{\delta}\leq v^{\delta}$, which is true for all $v-i,\delta\geq 0$.  The event $\{t<|V_1|\leq (v+\delta)/2\}$ has a probability which is bounded as follows:
\begin{eqnarray*}
   \Pr\left[t<|V_1|\leq (v+\delta)/2\right] & \leq & \sum_{i=\lfloor  t+1\rfloor}^{\lfloor(v+\delta)/2\rfloor} \binom{v}{i}\binom{v-i}{\delta}(1-p)^{i(v-i-\delta)} \\
   & \leq & \sum_{i=\lfloor  t+1\rfloor}^{\lfloor(v+\delta)/2\rfloor} \left(\frac{ev}{i}\right)^iv^{\delta}(1-p)^{i(v-\delta)/2} \\
   & \leq & v^{\delta}\sum_{i=\lfloor  t+1\rfloor}^{\lfloor(v+\delta)/2\rfloor} \left[ev(1-p)^{(v-\delta)/2}\right]^i
\end{eqnarray*}
If $v$ is large enough, then $ev(1-p)^{(v-\delta)/2}<ev(1-p)^{(v-pv-2\sqrt{2\ln v})/2}<1$.  Using this in our calculation, along with the bound $1-p\leq e^{-p}$,
\begin{eqnarray}
   \Pr\left[t<|V_1|\leq (v+\delta)/2\right] & \leq & 
   v^{\delta}\sum_{i=\lfloor  t+1\rfloor}^{\lfloor(v+\delta)/2\rfloor} \left[ev(1-p)^{(v-\delta)/2}\right]^i \nonumber \\
   & \leq & v^{\delta+1}\left[ev(1-p)^{(v-pv-2\sqrt{2\ln v})/2}\right]^t \nonumber \\
   & \leq & v^{\delta+1+t}e^t\exp\left\{-\frac{pvt}{2}+\frac{p^2vt}{2}+pt\sqrt{2\ln v}\right\} \nonumber \\
   & \leq & v^ve^v\exp\left\{-\frac{p(1-p)vt}{2}+pv\sqrt{2\ln v}\right\} \nonumber \\
   & = & \exp\left\{v\ln v+v+pv\sqrt{2\ln v}-\frac{p(1-p)}{2}vt\right\} \label{eq:appB}
\end{eqnarray}
Since $t\geq p(1-p)v-5\sqrt{v\ln v}$, the expression in (\ref{eq:appB}) is easily bounded above by $\exp\{-2\ln v\}$ for $v$ large enough. 

Summarizing, if $v$ is large enough, then with probability at least $1-3v^{-2}$, there is no set $S$ of size less than $\delta$ such that $V(G)-S$ is disconnected.
\epf

%%%%%%%%%%%%%%%%%%%%%%%%%%%%%%%%%%%%%%%%%%

\end{document}